\documentclass[11pt]{amsart}%
\usepackage{amscd,amsmath,latexsym,amsthm,amsfonts,amssymb,graphicx,color,geometry,hyperref}

\usepackage[utf8]{inputenc}
\usepackage{amsmath}
\usepackage{amsfonts}
\usepackage{amssymb}
\usepackage{graphicx}
\usepackage[all]{xy}
\usepackage{indentfirst}
\usepackage{graphicx}
\usepackage{graphics}
\usepackage{pict2e}
\usepackage{epic}
\usepackage{array} 
\numberwithin{equation}{section}
\usepackage{dsfont}
\providecommand{\U}[1]{\protect\rule{.1in}{.1in}}
\theoremstyle{plain}
\newtheorem{thm}{Theorem}[section]

\newtheorem{prop}[thm]{Proposition}
\theoremstyle{definition}

\newtheorem{?}[thm]{Problem}

\theoremstyle{definition}
\newtheorem*{nt*}{Notation}

\newcommand {\C} {\mathbb C}

\begin{document}
\title[]{Cyclicity of composition operators on the Paley-Wiener spaces}

\author{Pham Viet Hai, Waleed Noor and Osmar Reis Severiano}%
\address[P. V. Hai]{Faculty of Mathematics and Informatics, Hanoi University of Science and Technology, Khoa Toan-Tin, Dai hoc Bach khoa Hanoi, 1 Dai Co Viet, Hanoi, Vietnam.}%
\email{hai.phamviet@hust.edu.vn}
\address[W. Noor]{IMECC, Universidade Estadual de Campinas, Campinas-SP, Brazil}
\email{waleed@unicamp.br}
\address [O. R. Severiano]{ IMECC, Universidade Estadual de Campinas, Campinas, Brazil}
\email{osmar.rrseveriano@gmail.com}
\thanks{P.V. Hai is funded by Vietnam National Foundation for Science and Technology Development (NAFOSTED) under grant number 101.02-2021.24, O.R. Severiano is a postdoctoral fellow at the Programa de Matemática
and is supported by UNICAMP (Programa de Pesquisador de Pós-Doutorado PPPD) }

	
\begin{abstract} In this article we characterize the cyclicity of bounded composition operators $C_\phi f=f\circ \phi$ on the Paley-Wiener spaces of entire functions $B^2_\sigma$ for $\sigma>0$. We show that $C_\phi$ is cyclic precisely when $\phi(z)=z+b$ where either $b\in\mathbb{C}\setminus\mathbb{R}$ or $b\in\mathbb{R}$ with $0<|b|\leq \pi/\sigma$. We also describe when the reproducing kernels of $B^2_\sigma$ are cyclic vectors for $C_\phi$ and see that this is related to a question of completeness of exponential sequences in $L^2[-\sigma,\sigma]$. The interplay between cyclicity and complex symmetry plays a key role in this work.
\end{abstract}
	
\maketitle

\section{Introduction}

A bounded linear operator $T$ on a separable Hilbert space $\mathcal{H}$ is \emph{cyclic} with \emph{cyclic vector} $f\in\mathcal{H}$ if the linear span of the \emph{orbit} Orb$(f,T)=\{T^nf:n=0,1,2,\ldots\}$ is dense in $\mathcal{H}$. Similarly $T$ is \emph{supercyclic} if all scalar multiples of elements in Orb$(f,T)$ are dense in $\mathcal{H}$, and \emph{hypercyclic} if the orbit itself is dense. Cyclicity is a central theme in linear dynamics and has been studied widely due to its connection with the Invariant Subspace Problem (ISP). See the recent monographs \cite{Bayart-Matheron} and \cite{Linear Chaos} to learn more about linear dynamics. 

On the other hand a bounded linear operator $T$ on $\mathcal{H}$ is \emph{complex symmetric} if there exists an orthonormal basis for $\mathcal{H}$ with respect to which $T$ has a self-transpose matrix representation. An equivalent definition also exists. A \emph{conjugation} is a conjugate-linear operator $J:\mathcal{H}\to\mathcal{H}$ that satisfies the conditions \\

(a) $J$ is \emph{isometric}: $\langle Jf, Jg \rangle=\langle g,f \rangle$ $\forall$ $f,g\in\mathcal{H}$,\\

(b) $J$ is \emph{involutive}: $J^2=I$.\\ \\
We say that $T$ is $J$-\emph{symmetric} if $JT=T^*J$, and complex symmetric if there exists a conjugation $J$ with respect to which $T$ is $J$-symmetric. Complex symmetric operators are generalizations of complex symmetric matrices and of normal operators, and their study was initiated by Garcia, Putinar and Wogen \cite{Garcia 1, Garcia 2, Garcia 3, Garcia 4}.  If $JT=T^*J$ for some operator $T$ and conjugation $J$, then $T$ is cyclic (supercyclic) if and only if $T^*$ is cyclic (supercyclic). The conjugation $J$ serves as a bijection between the cyclic vectors of $T$ and $T^*$.

Let $\mathcal{S}$ be a  space of functions defined on a set $\Omega.$ 
A \emph{composition operator} $C_{\phi}$
on $\mathcal{S}$ with \emph{symbol} $\phi:\Omega\to\Omega$  is defined as
\[
C_{\phi}f=f\circ \phi, \quad f\in \mathcal{S}.
\]
The cyclicity of composition operators has been studied on several holomorphic function spaces (see \cite{Bayart, Gallardo, Noor}) and has become an active research topic within linear dynamics. In \cite{chacon2007composition}, Chácon, Chácon and Giménez initiated the study of composition operators on the classical Paley-Wiener space $B^2_{\pi}.$ They prove that $C_\phi$ is bounded on $B^2_{\pi}$ precisely when
\begin{align}\label{b1}
\phi(z)=az+b,\ \text{where} \ a\in \mathbb{R} \ \text{with} \ 0<|a|\leq 1 \ \text{and} \ b\in \mathbb{C}.
\end{align}
More recently Ikeda, Ishikawa and Yoshihiro \cite{Ikeda} show that this is true even in the general context of reproduction kernel Hilbert spaces of entire functions on $\mathbb{C}^n$. 

The cyclicity and complex symmetry of $C_\phi$ on the Hardy-Hilbert space $H^2(\mathbb{C}_+)$ of the half-plane were characterized by  Noor and Severiano \cite{Noor} for affine symbols $\phi$. In this article we study the cyclicity and complex symmetry of composition operators on the \textit{Paley-Wiener spaces} $B^2_{\sigma}$ for all $\sigma>0$. From results in \cite{Ikeda} it follows that the only bounded composition operators $C_\phi$ on $B^2_{\sigma}$ for $\sigma>0$ are those induced by symbols of the form \eqref{b1}. The Paley-Wiener spaces $B^2_\sigma$ are isometrically embedded into $H^2(\mathbb{C}_+)$ as so-called \emph{model subspaces} $K_\Theta$ (see \cite[page 305]{Nikolski}) defined as
\[
K_\Theta:=H^2(\mathbb{C}_+)\ominus\Theta H^2(\mathbb{C}_+) \ \ \mathrm{where} \ \ \Theta(z)=e^{i\sigma z}.
\]
Therefore one can view this work as the study of $C_\phi$ on some model subspaces of $H^2(\mathbb{C}_+)$. In contrast with $H^2(\mathbb{C}_+)$, our results show the existence of non-normal complex symmetric $C_\phi$ (see Theorem \ref{normal}), and that the cyclicity of $C_\phi$ and its adjoint $C_\phi^*$ in $B^2_\sigma$ depends on $\sigma>0$ (see Theorem \ref{cyclic}). In particular, we show that no $C_\phi$ is supercyclic on any $B^2_\sigma$. Finally we characterizese the reproducing kernels $(k_w)_{w\in\mathbb{C}}$ in $B^2_\sigma$ that are cyclic vectors for $C_\phi$ and show that this is closely related to the completeness of exponential sequences $(e^{i\lambda_n t})_{n\in\mathbb{N}}$ in $L^2[-\sigma,\sigma]$ (see Theorem \ref{cyclic vector}). The latter is a central question in non-harmonic Fourier analysis (see \cite{young}). Some of the main results are summarized in the following table.

\begin{table}[h]
\centering
\caption{Main results for $C_\phi$ on $B^2_\sigma$ where $\phi(z)=az+b$.}
\label{tab:example}
\begin{tabular}{|c|c|} 

\hline
$C_\phi$ cyclic & $a=1$ with $b\in\mathbb{C}\setminus\mathbb{R}$ or $0<|b|\leq\pi/\sigma$, $b\in\mathbb{R}$ \\
\hline
$C^*_\phi$ cyclic & $0<|a|<1$, $b\in\mathbb{C}$ or $C_\phi$ cyclic \\
\hline
$C_\phi$ supercyclic & Never \\
\hline
$C_\phi$ complex symmetric & $a=1$, $b\in\mathbb{C}$ or $a=-1$, $b\in\mathbb{C}$  \\ 
\hline
$C_\phi$ normal & $a=1$, $b\in\mathbb{C}$ or $a=-1$, $b\in\mathbb{R}$  \\

\hline
\end{tabular}
\end{table}
Before we begin, it is necessary to mention an error in \cite{chacon2007composition} during the computation of an adjoint formula for $C_\phi$ on $B^2_\pi$ which then leads to an incomplete description of normal $C_\phi$. This is discussed and corrected in Section 2.

\section{Preliminaries}
An entire function $f$ is of \textit{exponential type} if the inequality $
|f(z)|\leq Ae^{B|z|}$ holds for all $z\in \mathbb{C}$
and for some constants $A,B>0$. The exponential type $\sigma$ of $f$ is defined as the infimum of all $B>0$ for which this inequality holds and can be determined by 
\begin{align*}
\sigma=\limsup\limits_{r\rightarrow\infty}\frac{\log M_f(r)}{r} \ \  \mathrm{where} \ \ M_f(r)=\max_{|z|=r}|f(z).
\end{align*}

\subsection{Paley-Wiener spaces \texorpdfstring{$B^2_{\sigma}$}{lg}}  For $\sigma>0,$ the \emph{Paley-Wiener space} $B_\sigma^2$ consists of all entire functions of exponential type $\leq\sigma$ whose restrictions to $\mathbb{R}$ belong to $L^2(\mathbb{R}).$	The space $B_\sigma^2$ is a  reproducing kernel Hilbert space when endowed with the norm
\begin{equation*}
\left\| f\right\|^2 = \int^{\infty}_{-\infty}|f(t)|^2\mathrm{d}t, \quad f\in B^2_{\sigma}.
\end{equation*}
For each $w\in \mathbb{C},$ let $k_w$ denote the reproducing kernel for $B_{\sigma}^2$ at $w$ defined by
\begin{equation*}
k_w(z)=\frac{\sin \sigma(z-\overline{w})}{\pi(z-\overline{w})}, \quad z\in \mathbb{C}
\end{equation*}
and which satisfies the basic relation $\left\langle f,k_w\right\rangle =f(w)$ for all $f\in B_{\sigma}^2.$ The Paley-Wiener Theorem states that any $f\in B^2_\sigma$ can be represented as the \emph{inverse} Fourier transform 
\[
f(z)=(\mathfrak{F}^{-1}F)(z):=\frac{1}{\sqrt{2\pi}}\int_{-\sigma}^\sigma F(t)e^{izt}dt
\]
of some $F\in L^2[-\sigma,\sigma]$ where the Fourier transform $\mathfrak{F}:B^2_\sigma\to L^2[-\sigma,\sigma]$ is an isometric isomorphism. We denote $\widehat{f}:=\mathfrak{F}f$ for $f\in B^2_\sigma$ and in particular $(\widehat{k_w})(t)=e^{-i\bar{w}t}$ for $w\in\mathbb{C}$.
The monographs \cite{Branges} and \cite{young} may be consulted for more information about these spaces.
\subsection{Composition operators}
The composition operator $C_{\phi}$ is bounded on $B^2_{\sigma}$ if and only if  its inducing symbol $\phi$ has the form 
\begin{align}\label{symbol}
\phi(z)=az+b,\ \text{where} \ a\in \mathbb{R} \ \text{with} \ 0<|a|\leq 1 \ \text{and} \ b\in \mathbb{C}
\end{align}	
with the usual action of $C_{\phi}^*$ on reproducing kernels determined by $C_{\phi}^*k_w=k_{\phi(w)}$ for $w\in \mathbb{C}.$ For each $n\in \mathbb{N},$ the $n$-iterate of the self-map $\phi:\mathbb{C}\rightarrow \mathbb{C}$ is denoted by $\phi^{[n]}$ where
\begin{align}\label{iterate}
\phi^{[n]}(z)= \begin{cases}z+n b, & \text{if} \ a=1, \\ 
a^{n} z+\frac{\left(1-a^{n}\right)}{1-a}b,  & \text{if} \ a \neq 1 
\end{cases}
\end{align}	
and $C_{\phi}^n=\C_{\phi^{[n]}}.$ We see from \eqref{iterate} that if $0<|a|<1$ and $\alpha:=\frac{b}{1-a},$ then $\alpha$ is an attractive fixed point of $\phi$, that is,  $\phi^{[n]}(z)\rightarrow \alpha$ as $n\rightarrow \infty$ for all $z\in \mathbb{C}.$

\subsection{Adjoints of Composition operators} In \cite{chacon2007composition} the authors show that $C_\phi$ and $C_\phi^*$ on $B^2_\pi$ are unitarily equivalent, via the Fourier transform $\mathfrak{F}$, to a pair of weighted composition operators $\widehat{C}_{\phi}$ and $\widehat{C}_{\phi}^*$ on $L^2[-\pi,\pi]$ respectively. For simplicity they assume $a>0$ which is sufficient for their results. However, as we shall see in the next section, the cases $a>0$ and $a<0$ lead to very distinct outcomes for the complex symmetry and cyclicity of $C_\phi$. Therefore for the sake of completeness we provide a proof of their result for all $0<|a|\leq 1$.

\begin{prop}\label{adjoint formula} If $\phi(z)=az+b$ where $a\in \mathbb{R}$ with $0<|a|\leq 1$, $b\in \mathbb{C},$ then $C_{\phi}$ on $ B^2_{\sigma}$ is unitarily equivalent to a weighted composition operator $\widehat{C}_{\phi}$ on $L^2[-\sigma,\sigma]$ defined by 
\[(\widehat{C}_{\phi}F)(t)=\frac{1}{|a|}\chi_{(-|a|\sigma, |a|\sigma)}(t)e^{\frac{\mathrm{i}bt}{a}}F\left(\frac{t}{a}\right)\] 
where $\chi_{(c,d)}$ denotes a characteristic function. Moreover we have $(\widehat{C}_{\phi}^*F)(t)=\overline{e^{ibt}}F(at).$
\end{prop}
\begin{proof} We first assume $0<a\leq 1.$ For each $f\in B^2_{\sigma},$ we have
\begin{align}
(C_{\phi}f)(z)&=f(az+b)=\frac{1}{\sqrt{2\pi}}\int^\sigma_{-\sigma}\widehat{f}(t)e ^{i(az+b)t}dt\nonumber\\
&=\frac{1}{\sqrt{2\pi}}\int^\sigma_{-\sigma}\widehat{f}(t)e ^{iazt}e^{ibt}dt=\frac{1}{\sqrt{2\pi}}\int^{a\sigma}_{-a\sigma}\frac{1}{a}\widehat{f}\left(\frac{s}{a}\right)e ^{isz}e^{ib\left(\frac{s}{a}\right)}ds\nonumber\\
&=\frac{1}{\sqrt{2\pi}}\int^{\sigma}_{-\sigma}\frac{1}{a}\chi_{(-a\sigma, a\sigma)}(t)e^{\frac{ibs}{a}}\widehat{f}\left(\frac{s}{a}\right)e ^{isz}ds\nonumber\\
&=\frac{1}{\sqrt{2\pi}}\int^{\sigma}_{-\sigma}(\widehat{C}_{\phi}\widehat{f})(s)e ^{isz}ds=(\mathfrak{F}^{-1}\widehat{C}_\phi\mathfrak{F}f)(z)\nonumber
\end{align}
which gives $\mathfrak{F}C_{\phi}= \widehat{C}_{\phi}\mathfrak{F}$. For $-1\leq a <0$, first consider the symbol $\eta(z)=-z$. Then the simple change of variables $s=-t$ again gives
\begin{align}\label{eqc3}
(C_{\eta}f)(z)&=f(-z)=\frac{1}{\sqrt{2\pi}}\int^{\sigma}_{-\sigma}(\widehat{C}_{\eta}\widehat{f})(s)e ^{isz}ds=(\mathfrak{F}^{-1}\widehat{C}_\eta\mathfrak{F}f)(z)
\end{align}
and hence $\mathfrak{F}C_{\eta}= \widehat{C}_{\eta}\mathfrak{F}.$  Now let $\psi(z)=-az+b$ and note that $C_{\phi}=C_{\eta}C_{\psi}.$ Therefore 
\begin{align*}
\widehat{C}_{\eta}\widehat{C}_{\psi}=(\mathfrak{F}C_{\eta}\mathfrak{F}^{-1})(\mathfrak{F}C_{\psi}\mathfrak{F}^{-1})=\mathfrak{F}C_{\phi}\mathfrak{F}^{-1}
\end{align*}
and to verify that indeed $\widehat{C}_\phi=\widehat{C}_{\eta}\widehat{C}_{\psi}$, one easily sees that
\begin{align*}
(\widehat{C}_{\eta}\widehat{C}_{\psi}F)(t)&=\frac{1}{-a}e^{-\frac{ibt}{-a}}\chi_{(a\sigma, -a\sigma)}(-t)F\left(\frac{-t}{-a}\right)
=(\widehat{C}_{\phi}F)(t)
\end{align*}
for all $F\in L^2[-\sigma,\sigma]$. Therefore $\mathfrak{F}C_{\phi}= \widehat{C}_{\phi}\mathfrak{F}$ for all $-1\leq a<0$ as well. Since the Fourier transform $\mathfrak{F}:B^2_\sigma\to L^2[-\sigma,\sigma]$ is an isometric isomorphism, the proof is complete. The adjoint 
formula follows by the change of variables $t=s/a$ as follows:
\[
\langle \widehat{C}_{\phi}^*F, G\rangle=\int^\sigma_{-\sigma}\overline{e^{ibt}}F(at)\overline{G(t)}dt=\int^{|a|\sigma}_{-|a|\sigma}F(s)\overline{\frac{1}{|a|}e^{\frac{\mathrm{i}bs}{a}}G\left(\frac{s}{a}\right)}ds=\langle F, \widehat{C}_{\phi}G\rangle
\]
for all $F,G\in L^2[-\sigma,\sigma]$. This completes the proof of the result.
\end{proof}
In \cite[p. 2209]{chacon2007composition} it is claimed that if $\phi=az+b$ with $0<a<1,$ then the adjoint of $C_\phi$ is 
\begin{align*}
(C_{\phi}^*f)(z)=\frac{1}{a}f\left(\frac{z-\overline{b}}{a}\right), \quad f\in B^2_{\pi}.
\end{align*}
However this cannot be correct since if $f$ has exponential type $\pi$ then $f(z/a)$ must have type $\pi/a>\pi$. Hence $C_\phi^*$ is not well-defined on $B^2_{\pi}$. This unfortunately leads to an incomplete description of normal composition operators (compare \cite[Prop. 2.6]{chacon2007composition} with Table $1$).

\section{Cyclicity and complex symmetry of  \texorpdfstring{$C_{\phi}$}{lg}}\label{s6}
One of the main leitmotifs of this article is to show how results about cyclicity can lead to results about complex symmetry and vice versa. This hinges on a simple observation. If $TJ=JT^*$ for some operator $T$ and conjugation $J$, then $T$ is cyclic if and only if $T^*$ is cyclic. The conjugation $J$ serves as a bijection between the cyclic vectors of $T$ and $T^*$. We demonstrate this by first dealing with the case when $0<|a|<1$.

\begin{prop}\label{case a<1} Let $\phi(z)=az+b$ where $a\in \mathbb{R}$, $0<|a|<1$ and $b\in\mathbb{C}$. Then $C_\phi$ is not cyclic whereas $C_\phi^*$ is cyclic on $B^2_\sigma$ for $\sigma>0$. Hence $C_\phi $ is not complex symmetric.
\end{prop}
\begin{proof} We know that $C_\phi$ on $B^2_\sigma$ is unitarily equivalent to $\widehat{C}_\phi$ on $L^2[-\sigma,\sigma]$ by Proposition \ref{adjoint formula}. So we show that $\widehat{C}_\phi$ is not cyclic. For any $F\in L^2[-\sigma,\sigma]$ we have $\widehat{C}_\phi F=\chi_{(-|a|\sigma, |a|\sigma)}G$ for some $G\in L^2[-\sigma,\sigma]$. So every element in $\mathrm{span}\{\widehat{C}_\phi^n F:n=0,1\ldots\}$ must have the form $cF+H$ where $H$ vanishes ouside $(-|a|\sigma, |a|\sigma)$. In other words, the closure of this span consists only of functions that coincide with a constant multiple of $F$ outside $(-|a|\sigma, |a|\sigma)$. This clearly makes the cyclicity of $\widehat{C}_\phi$ and hence of $C_\phi$ impossible. For $C^*_\phi$ we show that every reproducing kernel $k_w$ is a cyclic vector. This is because $(C^{*}_\phi )^nk_w=C^*_{\phi^{[n]}}k_w=k_{\phi^{[n]}(w)}$ and $\phi^{[n]}(w)$ converges to the attractive fixed point $\alpha:=\frac{b}{1-a}$ of $\phi$. So any $f\in B^2_\sigma$ orthogonal to the orbit of $k_w$ under $C^*_\phi$ must vanish on a sequence with a limit point and hence $f\equiv 0$. Therefore the span of the orbit of every $k_w$ must be dense in $B^2_\sigma$ and hence $C_\phi^*$ is cyclic. It then follows that $C_\phi$ is not complex symmetric by the discussion above. 
\end{proof}

By Proposition \ref{case a<1} we need only consider the cases $a=\pm1$ for the following result. Note that in both cases ($J_af)(z)=\overline{f(-a\bar{z})}$ defines a conjugation on $B^2_\sigma$ for $\sigma>0$.

\begin{prop}\label{normal}Let $\phi(z)=az+b$ where $a=\pm1$ and $b\in \mathbb{C}$. Then $C_{\phi}$ on $B^2_{\sigma}$ is
\begin{enumerate}
\item always complex symmetric,
\item normal if and only if $a=1$, $b\in\mathbb{C}$ or $a=-1, b\in \mathbb{R}$,
\item self-adjoint if and only if $a=1,b\in i\mathbb{R}$ or $a=-1,b\in \mathbb{R},$
\item unitary if and only if $b\in \mathbb{R}.$
\end{enumerate}
Moreover $C_\phi$ is $J_a$-symmetric on $B^2_{\sigma}$.
\end{prop}
\begin{proof} We first note that when $a=-1$ we have $C^2_\phi=I$ and hence $C_\phi$ is complex symmetric since every operator that is algebraic of order $2$ is complex symmetric by \cite[Thm. 2]{Garcia 1}. In general by Proposition \ref{adjoint formula} we have \[
(\widehat{C}_\phi F)(t) =e^{\frac{ibt}{a}}F(t/a) \ \ \ \mathrm{and} \ \ \ (\widehat{C}^*_\phi F)(t) =e^{\overline{ibt}}F(at)
\]
for $F\in L^2[-\sigma,\sigma]$. So when $a=1$ we see that $\widehat{C}_\phi=M_{e^{ibt}}$ is a multiplication operator which is normal. This shows that $C_\phi$ is always complex symmetric and gives $(1)$. If $a=-1$, then $\widehat{C}_\phi \widehat{C}_\phi^*=\widehat{C}_\phi ^*\widehat{C}_\phi$ precisely when $e^{-ibt}e^{-\overline{ibt}}=e^{\overline{ibt}}e^{ibt}$ or when $e^{-i2\mathrm{Im}(b)t}=e^{i2\mathrm{Im}(b)t}$, and this holds only when $b\in\mathbb{R}$. This gives $(2)$ and one easily obtains $(3)$ and $(4)$ similarly. Finally we show that $C_\phi$ is $J_a$-symmetric by first noting that
\[
(J_a C_\phi J_af)(z)=\overline{(C_\phi J_af)(-a\bar{z})}=\overline{(J_af)(-\bar{z}-a\bar{b})}=f(az-a\bar{b})=(C_\psi f)(z)
\]
where $\psi(z)=az-a\bar{b}$. We claim that $C_\psi=C_\phi^*$ and this follows by Proposition \ref{adjoint formula} because \[(\widehat{C}_{\psi}F)(t)=e^{-i\bar{b}t}F(t/a)=\overline{e^{ibt}}F(at)=(\widehat{C}_{\phi}^*F)(t)\] for all $F\in L^2[-\sigma,\sigma]$. Therefore $J_a C_\phi J_a=C_\phi^*$ and completes the proof of the result.
\end{proof}

Our main result characterizes the cyclicity of $C_\phi$ and $C_\phi^*$ simultaneously using complex symmetry, and then shows that no $C_\phi$ is supercyclic on $B^2_\sigma$ for all $\sigma>0$ via normality. The latter was proved for the case $\sigma=\pi$ in \cite[Thm. 2.7]{chacon2007composition} using a lengthier argument.

\begin{thm}\label{cyclic}Let $\phi(z)=az+b$ where $a\in\mathbb{R}$, $0<|a|\leq 1$ and $b\in \mathbb{C}$. Then $C_{\phi}$ on $B^2_\sigma$ is
\begin{enumerate}
\item cyclic if and only if $a=1$ with $b\in\mathbb{C}\setminus\mathbb{R}$ or $0<|b|\leq\pi/\sigma$, $b\in\mathbb{R}$, and
\item never supercyclic.
\end{enumerate}
Moreover $C_\phi^*$ is cyclic if and only if $0<|a|<1$ or when $C_\phi$ is cyclic.
\end{thm}
\begin{proof} The case $0<|a|<1$ has been dealt with in Proposition \ref{case a<1} showing that $C_\phi$ is not cyclic or supercyclic. So first note that when $a=-1$ then $C_\phi^2f=f$ and hence $C_\phi$ cannot be cyclic or supercyclic as the orbit of any $f\in B^2_\sigma$ contains at most two elements $f$ and $f\circ\phi$. So let $a=1$. Then $C_\phi$ is normal and normal operators are never supercyclic \cite[p. 564]{Hilden}. Therefore this proves $(2)$. For cyclicity we first observe that $\widehat{C}_\phi=M_{e^{ibt}}$ is a multiplication operator on $L^2[-\sigma,\sigma]$. A consequence of the spectral theory for normal operators is that a multiplication operator $M_\Phi$ on $L^2(\mu)$ (where $\Phi\in L^\infty(\mu)$ and $\mu$ is a compactly supported measure on $\mathbb{C}$) is cyclic if and only if $\Phi$ is injective on a set of full measure (see \cite[Thm. 1.1 and Prop. 1.3]{Ross}). Therefore we need to determine when $\Phi(t)=e^{ibt}$ is injective almost everywhere on $[-\sigma,\sigma]$. For $b\in\mathbb{C}\setminus\mathbb{R}$ we see that $|\Phi(t)|=e^{-\mathrm{Im}(b)t}$ which is clearly injective on all of $[-\sigma,\sigma]$ and hence so is $\Phi$. For the case when $b\in\mathbb{R}$ note that $\Phi$ is $2\pi/|b|$ periodic. It follows that $\Phi$ is injective on $(-\sigma,\sigma)$ (hence a.e on $[-\sigma,\sigma]$) precisely when the period is greater than or equal to the length of the interval, that is, precisely when
\[
\frac{2\pi}{|b|}\geq 2\sigma \ \ \ \mathrm{equivalently} \ \ \ 0<|b|\leq\frac{\pi}{\sigma}.
\]
The assertion about cyclicity of $C_\phi^*$ follows by Proposition \ref{case a<1} and complex symmetry. 
\end{proof}
Finally we ask the following question: \emph{Are any of the reproducing kernels $(k_w)_{w\in\mathbb{C}}$ cyclic vectors for $C_\phi$}? The answer reveals an interesting dichotomy and a connection with the question of completeness of exponential sequences in $L^2[-\sigma,\sigma]$. 

\begin{thm}\label{cyclic vector}Let $\phi(z)=z+b$ where $b\in \mathbb{C}$. Then the kernels $(k_w)_{w\in\mathbb{C}}$ in $B^2_\sigma$ are 
\begin{enumerate}
\item all cyclic vectors for $C_\phi$ when $b\in\mathbb{C}\setminus\mathbb{R}$ or $0<|b|<\pi/\sigma$, $b\in\mathbb{R}$, and
\item never cyclic for $C_\phi$ when $|b|=\pi/\sigma$. 
\end{enumerate}
\end{thm}
\begin{proof} We first consider the case when $b\in\mathbb{C}\setminus\mathbb{R}$. Note that since $\phi^{[n]}(z)=z+nb$, one easily sees that $C_{\phi}^nk_w=k_{w-n\bar{b}}$ for $w\in\mathbb{C}$. So any function $f\in B^2_\sigma$ orthogonal to the span of the orbit of $k_w$ must vanish at $w_n:=w-n\bar{b}$ for all $n\geq 0$. But $(w_n)_{n\geq 0}$ does not satisfy the Blaschke-Type condition required for zero sets of functions in $B^2_\sigma$, that is
\[
\sum_{n\geq 0}\frac{|\mathrm{Im}(w_n)|}{1+|w_n|^2}\simeq \sum_{n\geq 0}\frac{n}{1+n^2}=\infty.
\]
Hence $f\equiv 0$ and span$(C_{\phi}^nk_w)_{n\geq 0}$ is dense in $B^2_\sigma$. So now let $b\in\mathbb{R}$. Since $(\mathfrak{F}k_w)(t)=e^{-i\bar{w}t}$ for $w\in\mathbb{C}$, it is enough to show that every exponential $e^{iwt}$ is a cyclic vector for $\widehat{C}_\phi=M_{e^{ibt}}$. This is equivalent to the completeness of the sequence $(e^{i(bn+w)t})_{n\geq 0}$ in $L^2[-\sigma,\sigma]$. But this sequence is the image of $(e^{ibnt})_{n\geq 0}$ under  $M_{e^{iwt}}$ which has dense range in $L^2[-\sigma,\sigma]$. Therefore it is sufficient to prove the completeness of $(e^{ibnt})_{n\geq 0}$ in $L^2[-\sigma,\sigma]$. For this we use a classical result of Carleman (see \cite[page 97]{young}): If $(\lambda_n)_{n\geq 0}$ is a sequence of positive real numbers and 
\[
\liminf_{n\to\infty}\frac{n}{\lambda_n}>\frac{\sigma}{\pi},
\]
then $(e^{i\lambda_nt})_{n\geq0}$ is complete in $\mathcal{C}[-\sigma,\sigma]$ (space of continuous functions). This condition is clearly satisfied for $\lambda_n:=bn$ when $0<b<\pi/\sigma$. The case $-\pi/\sigma<b<0$ follows since now $(e^{-ibnt})_{n\geq0}$ is complete and then conjugating. If $|b|=\pi/\sigma$, then the exponentials $(e^{ibnt})_{n\geq0}$ are $2\sigma$-periodic as we saw in the proof of Theorem \ref{cyclic}. It follows that $(e^{ibnt})_{n\in\mathbb{Z}}$ is an orthogonal basis for $L^2[-\sigma,\sigma]$. So $(e^{ibnt})_{n\geq0}$ in not complete in $L^2[-\sigma,\sigma]$ and neither is its image under the injective operator $M_{e^{iwt}}$. This completes the proof.
\end{proof}

The last result hints at the possibility of finding deeper connections between  composition operators on model subspaces $K_\Theta$ of $H^2(\mathbb{C}_+)$ and non-harmonic Fourier analysis.


\end{document}